\documentclass[12pt]{article}
\usepackage{amsmath,amsfonts,amsthm,latexsym, amssymb, wasysym}

\newtheorem{theorem}{Theorem}
\newtheorem{corollary}[theorem]{Corollary}
\newtheorem{proposition}[theorem]{Proposition}

\newtheorem{definition}[theorem]{Definition}
\newtheorem{axiom}[theorem]{Axiom}
\newtheorem{lemma}[theorem]{Lemma}

\newtheorem{statement}[theorem]{Statement}

\begin{document}
\title{Deciding the Continuum Hypothesis \\ with the Inverse Powerset}
\author{Patrick St-Amant}
\date{}

\maketitle

\begin{abstract}We introduce the concept of inverse powerset by adding three axioms to the Zermelo-Fraenkel set theory. This extends the Zermelo-Fraenkel set theory with a new type of set which is motivated by an intuitive meaning and interesting applications. We present different ways to extend the definition of cardinality and show that one implies the continuum hypothesis while another implies the negation of the continuum hypothesis. We will also explore the idea of empty sets of different cardinalities which could be seen as the empty counterpart of Cantor's theorem for infinite sets.\end{abstract}

\section{Introduction}
Relying on G\"{o}del's results, Cohen proved the independence of Cantor's continuum hypothesis from Zermelo-Fraenkel set theory with the axiom of choice \cite{Goedel, Cohen1, Cohen2}. Many axioms have been considered to extend Zermelo-Fraenkel set theory in a way which would permit a proof or a disproof of the continuum hypothesis. In particular, the axiom of constructibility implies the generalized continuum hypothesis \cite{Goedel} and the Freiling's axiom of symmetry is equivalent to the negation of the continuum hypothesis \cite{Freiling}. More recently, Woodin developed an extensive framework aimed at disproving the continuum hypothesis \cite{Woodin1, Woodin2}. Accepting or rejecting such axioms often becomes a matter of taste and `intuition', but each of those suggestions give rise to new systems of axioms which are worth studying by themselves or in relation to the others. In the following, we will formalize the concept of inverse powerset which suggests another approach to the continuum hypothesis. This view is analogous to numbers and as discussed in section \ref{motivation} it is motivated by intuition and usefulness.

The negative numbers have been defined to allow us to solve equations such as $2+y=1$. Similarly, the complex numbers have been introduced by Euler to consider solutions of equations such as $y^2=-1$. In the realm of set theory, in particular Zermelo-Fraenkel set theory (\textit{ZF}), we can ask a similar question and try to find a set $Y$ such that the statement  $P(Y)=X$ is true for some fixed set $X$. For example, if we take $X=\{\emptyset,\{a\},\{b\},\{a,b\}\}$, then we find that $Y=\{a,b\}$ satisfies the statement. But if we want a set $Y$ which would satisfy the statement $P(Y)=\{\emptyset,\{a\},\{a,b\}\}$, we have to define a new type of set in a way that is similar to the introduction of the negative and complex numbers.

One can wonder about the internal structure of $Y$ where $Y$ satisfies $P(Y)=X$ and where $X$ is not a powerset. It turns out that $Y$ can be seen as an object composed of subsets. Take $W$ to be a classical set of ZF, then it is always the case that a subset of a subset of $W$ is a subset of $W$. In our case, for $Y$, it is possible a subset of a subset of $Y$ is not a subset of $Y$. In itself, this could be seen as a generalization where a generalized set $Z$ is a classical Zermelo-Fraenkel set if each member of the subsets of $Z$ is a member of $Z$. Surprisingly, this generalization of sets fits perfectly with the concept of inverse powerset and also comes with a philosophical intuition. We will discuss this in details in the subsequent sections.

The goal of this paper is twofold, first, to introduce the concept of the inverse powerset by adding three axioms to a slightly modified Zermelo-Fraenkel set theory. We will give some arguments justifying the introduction of the concept of inverse powerset in set theory. Secondly, in this extended setting, we will show that it is possible to prove or disprove the continuum hypothesis by choosing the suitable extended definition of cardinality. Relying on this extended Zermelo-Fraenkel set theory, we will discuss an interesting avenue of research where we can consider empty sets of different cardinality in the same way that there are infinite sets of different cardinality.

In this paper, we will not prove the consistency of the extended Zermelo-Fraenkel set theory or the relative consistency of the three axioms with ZF, but we will see that introducing such axioms can be used to decide the continuum hypothesis. Having introduced new types of sets, we find that we have to generalize the definition of cardinality and we will see that one choice of generalization implies the continuum hypothesis while others imply the falsity of the continuum hypothesis. The question of consistency is a serious question which has to be answered eventually. It is the belief of the author that a framework including the concept of inverse powerset gives us a rich theory to further investigate the foundation of mathematics.

\section{Motivation}\label{motivation}

Before giving the formal axioms and definitions, we will give a few arguments which motivate the use of the concept of the inverse powerset with the intention of showing its usefulness and possibly even its unavoidability.

\subsection{Inverse Powerset and the Continuum Hypothesis}

The idea of inverse powerset first occurred in relation to the continuum hypothesis and it will be the main application presented in this paper. Informally, if there exists a set $B$ such that $|\mathbb{N}|<|B|<|P(\mathbb{N})|$ then `applying' the inverse powerset operator to those inequations we get $|P^{-1}(\mathbb{N})|<|P^{-1}(B)|<|P^{-1}(P(\mathbb{N}))|=|\mathbb{N}|$. This way, the question of the continuum hypothesis is reduce to choosing a definition of cardinality for objects of the type $|P^{-1}(\mathbb{N})|$. If a definition of cardinality implies $|P^{-1}(\mathbb{N})|=|\mathbb{N})|$, then by an adequate squeeze theorem, we have that $|P^{-1}(B)|=|\mathbb{N}|$ which will induce a contradiction and thus prove the continuum hypothesis. If we choose another definition of cardinality which implies $|P^{-1}(\mathbb{N})|<|\mathbb{N})|$, then, in a proper setting, we have the negation of the continuum hypothesis. This will be discussed in detail in section \ref{negcontinuum}.

A primary investigation suggests taking the following statements as axioms.

\begin{statement}[Inverse powerset]\label{axiomInverse}$$ \forall X \exists Y[ (P(Y)=X)]$$\end{statement}

We would denote $Y$ as $P^{-1}(X)$ and call it the \textit{inverse powerset} of $X$. Hence, for some fixed $X$, we introduce a new type of set $Y$ which satisfies the statement $P(Y)=X$.  Similarly to what is commonly done with the power of a set, we can consider $P^{-1}$ to be an operator. For example, the set $Y=\{a,b\}$ satisfies the statement $P(Y)=\{\emptyset,\{a\},\{b\},\{a,b\}\}$. This axiom can be seen as allowing a set theory which is `closed' under the powerset operator $P$.

We now give another statement which aims at expressing the notion that $P^{-1}$ is the inverse operator of $P$. This statement has the purpose of allowing the proof of useful propositions.

\begin{statement}[Invertibility]\label{axiomInvertibility}$$ \forall X [ P^{-1}(P(X))=X ]$$\end{statement}

If we take the set $\{\emptyset,\{a\},\{b\},\{c\},\{a,b\},\{a,c\},\{b,c\},\{a,b,c\}\}$ and apply $P^{-1}$ to it, we find $P^{-1}(\{\emptyset,\{a\},\{b\},\{c\},\{a,b\},\{a,c\},\{b,c\},\{a,b,c\}\})=P^{-1}(P(\{a,b,c\}))=\{a,b,c\}.$

Taken as axioms, we can prove most of the needed properties to construct a reliable system, but they do not inform us about the internal structure of inverse powerset. Moreover, how can we apply the powerset operator on a set $P^{-1}(X)$ with an unknown internal structure? In the following section, we present axioms which entail the above statements and which have the quality of being more elegant, in particular, the inverse powerset axiom has the flavor of being the dual of the powerset axiom.

\subsection{Duality}

As mentioned, the concept of inverse powerset can be approached with the statement $\forall X \exists Y[ (P(Y)=X)]$, but a closer investigation and the fact that the powerset must be well defined yields two axioms close to the following forms.

$$\forall X\exists Y\forall A[A\in_1 X\Leftrightarrow A\in Y],$$
$$\forall Y\exists X\forall A[A\in Y\Leftrightarrow A\in_1 X].$$

Here, `$\in_1$' is a new membership relation symbol and we read $(A\in_1 B)$ as $A$ is a subset-member of $B$, in particular, if $A$ and $B$ are sets of Zermelo-Fraenkel set theory then $\in_1$ is the subset relation $\subseteq$. Note that in the following sections, the sets which come from ZF will formally be defined as the elements of the Von Newmann universe and will be called Zermelo sets. We will explain in more detail in the subsequent sections. We will see that taken as axioms, those statements imply $ \forall X \exists Y[ (P(Y)=X)]$ and $ \forall X [ P^{-1}(P(X))=X ]$. Moreover, they give rise to the internal structure of objects of the form $P^{-1}(Y)$. It is probably a matter of taste, but there is a certain elegance which reminds one of duality and adjoint functors of category theory (Chapter 4 of \cite{Maclane}). In a personal communication, William Lawvere suggested that from a categorical point of view, the inverse powerset could also be seen in the setting of adjoints.

\subsection{Generalized Sets}\label{generalized sets}

We already mentioned that it is possible to give a generalization of sets, where a generalized set can be seen as only made out of subsets. For example, such a generalized set is the object which only have the subsets $\{a\}$ and $\{a,b\}$. In this context, a set of ZF is defined as a set such that each subset of the subsets of $X$ is also a subset of $X$. In contrast, the generalized set having $\{a\},\{a,b\}$ and $\{b\}$ as subsets can be considered to be a set of ZF. This is an interesting generalization in itself and surprisingly, it follows from the concept of inverse powerset.

\subsection{Parts and Collections of Objects}

There is an interpretation of the inverse powerset which is more rooted in our worldly intuition. Without a formal mathematical theory like ZF, the concept of a set is naturally accessible. Primary-grade children naturally have an intuition of a collection of objects \cite{suppes_psy} and this capacity could be seen as an inherent cognitive process. There is another cognitive process which could be seen as the counterpart of the process of collection. When we look at an object as a whole, we have the capacity to recognize parts of this object. From another point of view, this capacity can be seen as being able to insert this whole object in a universe where it can be seen as composed of smaller constituents. For example, take as objects the natural numbers $1, 2$ and $3$. We can make a few different sets containing these numbers, for example $\{1,3\}$, $\{2,3\}$, $\{3\}$, $\{1,2,3\}$ and so on. At first, natural numbers can be seen as `whole' numbers which cannot be divided. By introducing the operation of division and the concept of fractions, it is possible to take a natural number and put it in the universe of fractions. This natural number is not seen as a whole but as an entity which is composed of smaller parts, for example, $1$ can be seen as the sum of $\frac{1}{2}+\frac{1}{2}$ or $\frac{3}{4}+\frac{1}{4}$.
For another example, take 5 flowers of different colors. Out of those 5 flowers we can create a few different bouquets of different size, this is analogous to the powerset operator. For an analogy of the inverse powerset, take one flower, then there is an extended set U (or a universe) where this flower is considered to be a subset or composed of elements. For example, in the present case, the set U could be a set of distinct cells. In the case of a flower, there is a way to see a flower as a subset of certain `smaller' elements such that we can assigns types of elements to different parts of the flower. For example, there are different types of cells such as petal cells, stalk cells and so on. The fact that we can differentiate between different parts is important. This is exactly what the inverse powers set permits us to do, but a good understanding of this necessitates a close examination of the axioms of inverse powerset.

In our primary development of an extended set theory (as done in section \ref{EZF level 1}), we do not have objects of the form $P^{-1}(P^{-1}(X))$, but in section \ref{level n} we will define a setting which allows objects such as $P^{-1}(P^{-1}(X))$ and $P^{-1}(P^{-1}(P^{-1}(P^{-1}(X))))$. Extending such constructions indefinitely, we have an interesting philosophical idea that `everything' can be seen as being composed of different types of parts. This idea is closely related to important questions of the philosophical theory of mereology \cite{Varzi}.

\subsection{Logarithm and Extended Cardinal Arithmetic}

The inverse powerset can help in defining the concept of logarithm for cardinal numbers. As seen above, the inverse powerset can be seen as a set $Y$ satisfying the statement $P(Y)=X$. With the cardinal numbers framework, we can translate this into $2^{\mid Y\mid}=\mid X\mid$ and this is similar to saying that $\mid Y\mid$ is the `exponent' we have to give to $2$ to get $X$. In this case, it could be interesting to denote $\mid Y\mid$ as $\log_2(\mid X\mid)$. An interesting avenue of investigation would be to enrich the cardinal number arithmetic with logarithms and all arithmetic functions.

It is well known that the ordinals extend the natural numbers and that each natural number can be represented by a set. It would be interesting to have an extended set of ordinals which extend the real numbers, in particular this would mean that we could associate a set to each real number. A not-so distant goal would be to have a cardinal arithmetic which acts as the real numbers arithmetic.

There might also be interesting applications in topology where cardinal invariants play a very important role.

\subsection{Algebraically Closed Set Theory}

Extending ZF using the inverse powerset provides us a setting where the formula $P(Y)=X$ has a `solution' in $Y$ for any set $X$ from Zermelo-Fraenkel. We will see in section \ref{level n} that there is an extended setting where we have solutions to formulas $P(P(...P(X)...))=X$ where $P$ appears a finite number of times. Those are a few steps toward a set theory that is `algebraically closed' and this topic can be investigated further through the concept of existentially closed models.

\subsection{Empty Sets of Different Cardinalities}

Similarly to what is done for infinite sets, the concept of inverse powerset will permit us to investigate empty sets of different cardinalities. This will be discussed in more detail in section \ref{emptysets}.

\section{Extending Zermelo-Fraenkel}\label{EZF level 1}

We extend Zermelo-Fraenkel set theory by adding the concept of the inverse powerset. We will add three axioms to a slightly modified Zermelo-Fraenkel set theory and prove a few practical propositions.

\subsection{Inverse Powerset Axiom}

Similarly to what is done when passing from the real numbers to the complex numbers, we need to extend the class of the classical sets which arises from the Zermelo-Fraenkel set theory with the new types of sets. To describe the class of classical sets which comes from ZF, we will take $\mathcal{V}$ be the Von Newmann universe and call each member of $\mathcal{V}$ a \textit{Zermelo} set.

To the classical language of ZF, we add a subset membership relation symbol `$\in_1$'. A primitive atomic formula containing this symbol is of the form $(A\in_1 B)$ where $A$ and $B$ are either general variables or the constant $\emptyset$. We will read $(A\in_1 B)$ as $A$ is subset-member of $B$. Note that sets (Zermelo or not) can be represented by either an uppercase or lowercase letter.

Our extended system of axioms will be composed of the axioms of ZF without the Powerset axiom such that each axiom is over $\mathcal{V}$. In other words, each variable bounded by the quantifiers `$\forall$' and `$\exists$' take values from $\mathcal{V}$. To those axioms we add the three following axioms.

We now give the axioms giving rise to the powerset and inverse powerset.

\begin{axiom}[Powerset]$$\forall X(\exists Y\in\mathcal{V})(\forall A\in\mathcal{V})[A\in_1 X\Leftrightarrow A\in Y]$$\end{axiom}

We will denote $Y$ as $P(X)$ and call it the \textit{powerset} of $X$. If we replace `$\in_1$' by `$\subseteq$' in the previous axiom we have exactly the classical powerset axiom. We will see that the axiom \ref{subset assignment} below will provide this replacement when $X$ and $Y$ are a Zermelo set. Let $\mathcal{V}^*$ be the $\mathcal{V}$ without the empty set.

\begin{axiom}[Inverse powerset]$$(\forall Y\in\mathcal{V}^*) \exists X(\forall A\in\mathcal{V})[A\in Y\Leftrightarrow A\in_1 X]$$\end{axiom}

We will denote $X$ as $P^{-1}(Y)$ and call it the \textit{inverse powerset} of $Y$. For example, take $Y=\{a,b\}$, then by the inverse powerset axiom, there exists a set $X$ such that $a$ and $b$ are its subset-members. By the powerset axiom, there exists a set $Y$ such that $a$ and $b$ are its elements.

The following axiom is needed to make the subset-membership symbol `$\in_1$' correspond to the subset symbol `$\subseteq$' for the case of Zermelo sets. This axiom which makes the link between the classical subset symbol and the subset-membership symbol will be discussed in detail in the next section.

\begin{axiom}[Subset Assignment]\label{subset assignment}$$(\forall X\in\mathcal{V})(\forall A\in\mathcal{V})[\forall x(x\in A\Rightarrow x\in X)\Leftrightarrow (A\in_1 X)]$$\end{axiom}

Before going further, let's look at another example. Take $X=\{a,b\}$, then by the powerset axiom $P(X)=\{\emptyset,\{a\},\{b\},\{a,b\}\}$. Using the inverse powerset axiom we have that each element of $P(X)$ are the subset-members of $P^{-1}(P(X))$. Our intuition tells us that $P^{-1}(P(X))$ should in fact be a Zermelo set $X$. This will be proved in section \ref{propositions}.

\subsection{Formal System of Axioms}

We now formally give the extended system of axioms of Zermelo-Fraenkel including the inverse powerset axioms (written without the $P$ and $P^{-1}$ notations). To avoid possible inconsistencies, in particular with the axiom of foundation, we have restricted the axioms of ZF (except for the axiom of powerset) to the Von Newmann universe $\mathcal{V}$ of Zermelo sets. The bolded names of some axioms signify that it is a new axiom or is an axiom of ZF for which the formula was modified. Note that the non-bolded axioms are the axioms of ZF with the difference that each variable under the scope of the quantifiers `$\forall$' and `$\exists$' are taken to be over the Von Newmann universe $\mathcal{V}$. Note that some of the axioms can be proved to be redundant. Notice that we have extended the axiom of extensionality to include the sets which are not Zermelo, this will be discussed in detail in section \ref{section subsets}. Note that the formulation of axioms 7 and 8 are based on \cite{suppes}.

\begin{definition}The extended Zermelo-Fraenkel system is composed of the following axioms.
\begin{itemize}
\item[1.] Null Set:\hspace{1cm} $(\exists X \in\mathcal{V}) (\forall Y \in\mathcal{V})[\neg(Y\in X)].$
\item[2.] \textbf{Extensionality}:\hspace{1cm}$(\forall A\in\mathcal{V})(\forall B\in\mathcal{V})(\forall S\in\mathcal{V})$ $$[(S\in_1 A \Leftrightarrow S\in_1 B)\Rightarrow A=B].$$
\item[3.] Foundation:\hspace{1cm} $(\forall X \in\mathcal{V}) [ (\exists A \in\mathcal{V}) ( A \in X) \Rightarrow$ $$(\exists Y \in\mathcal{V}) ( Y \in X \land \lnot (\exists Z \in\mathcal{V}) (z \in Y \land z \in X))].$$
\item[4.] Pairing:\hspace{1cm}$(\forall A \in\mathcal{V})(\forall B \in\mathcal{V})(\exists C \in\mathcal{V})$ $$[\forall X \in\mathcal{V})(X\in C\Leftrightarrow (X= A\vee X=B)]$$
\item[5.] Union:\hspace{1cm} $(\forall \mathcal{F} \in\mathcal{V})\,(\exists A \in\mathcal{V}) \, (\forall Y \in\mathcal{V})\, (\forall X \in\mathcal{V})$ $$[(X \in Y \land Y \in \mathcal{F}) \Rightarrow X \in A].$$
\item[6.] Infinity:\hspace{1cm} $(\exists X \in\mathcal{V}) \left[\varnothing \in X \land (\forall Y \in\mathcal{V}) (Y \in X \Rightarrow S(Y)  \in X)\right]$.
\item[7.] Separation:\hspace{1cm} $(\forall Z \in\mathcal{V})(\exists Y \in\mathcal{V}) (\forall X \in\mathcal{V})$ $$[X \in Y \Leftrightarrow ( X \in Z \land \varphi(X) )].$$
\item[8.] Replacement:\hspace{1cm} $(\forall X \in\mathcal{V})(\forall Y \in\mathcal{V})(\forall Z \in\mathcal{V})$ $$[X\in A \land \varphi(X,Y) \land \varphi(X,Y)\Rightarrow Y=Z]
\Rightarrow$$ $$(\exists B \in\mathcal{V})(\forall Y \in\mathcal{V})[Y\in B \Leftrightarrow (\exists X \in\mathcal{V})(X\in A \land \varphi(X,Y))].$$
\item[9.] \textbf{Powerset}:\hspace{1cm} $\forall X(\exists Y\in\mathcal{V})(\forall A\in\mathcal{V})[A\in_1 X\Leftrightarrow A\in Y]$.
\item[10.] \textbf{Inverse Powerset}:\hspace{1cm} $(\forall Y\in\mathcal{V}^*) \exists X(\forall A\in\mathcal{V})[A\in Y\Leftrightarrow A\in_1 X]$.
\item[11.] \textbf{Subset Assignment}:\hspace{1cm} $(\forall X\in\mathcal{V})(\forall A\in\mathcal{V})$ $$[(x\in A\Rightarrow x\in X)\Leftrightarrow (A\in_1 X)].$$
\end{itemize}
\end{definition}

Here $\mathcal{V}$ is the Von Newmann universe and $\mathcal{V}^*$ is the Von Newmann universe without the empty set. As usual, in the powerset axiom 9, we will denote $Y$ as $P(X)$ and in the inverse powerset axiom 10, we will denote $X$ as $P^{-1}(Y)$.

We now have a universe of discourse where we can define the class of extended sets.

\begin{definition} We define the class $\text{EZF}$ of rational sets as follows:
\begin{itemize}
\item[1.] If $X\in\mathcal{V}$ then $X\in\text{EZF}$,
\item[2.] If $X\in\mathcal{V}$ then $P^{-1}(X)\in\text{EZF}$,
\end{itemize}
\end{definition}

Sets which are in $\text{EZF}$ but not in $\mathcal{V}$ will be called \textit{non-Zermelo} and the class of all such sets is written as $\text{EZF}\setminus\mathcal{V}$. Examples of non-Zermelo sets are $P^{-1}(\{1,2,4\})$, $P^{-1}(\{\{1,2,4\}\})$, $P^{-1}(A)$, $P^{-1}(\{A,B\})$ and $P^{-1}(\{B,\{A,B\}\})$ for any sets $A,B$ and where $0 = \{\}, 1=\{0\}, 2=\{0,1\}$ and so on.

\subsection{Subsets}\label{section subsets}

We will now describe in more detail the subsets of a non-Zermelo set such as $P^{-1}(X)$. An important idea behind the extended definition of subsets is to have

\begin{center}$P^{-1}(X')\subseteq P^{-1}(X)$ if and only if $X'\subseteq X$.\end{center}

\noindent In analogy with the negative numbers, having defined the negative numbers $-1$ and $-2$ by using equations such as $2+y=1$ and $3+y=1$, we need to order them. By the axioms of powerset and inverse powerset, we have that the non-Zermelo sets are composed of subset-elements.  As discussed in section \ref{generalized sets}, in EZF, we need to shift our focus from the elements of a set to the subsets of a set which seem more general and fundamental. This way, a Zermelo set is defined as a set such that each subset of the subsets of $X$ is also a subset of $X$.

Now, recall that the classical definition of a subset is $$A\subseteq B\Leftrightarrow \forall x(x\in A\Rightarrow x\in B).$$ Again, we give the subset assignment axiom which includes the classical definition.

\begin{axiom}[Subset Assignment]\label{subset assignment}$$(\forall X\in\mathcal{V})(\forall A\in\mathcal{V})[\forall x(x\in A\Rightarrow x\in X)\Leftrightarrow (A\in_1 X)]$$\end{axiom}

For Zermelo sets, `$\in_1$' corresponds to `$\subseteq$' by this axiom. Recall that the classical definition of a subset is $A\subseteq B\Leftrightarrow \forall x(x\in A\Rightarrow x\in B)$.

We now give the definition of subset which will allows proofs of useful propositions and ensures that the new type of sets are also ordered via the subset symbol `$\subseteq$'.

\begin{definition}[Subsets]\label{extsubset} $X\subseteq Y$ if and only if
\medskip
\begin{center}$(\forall A\in\mathcal{V})(A\in X\Rightarrow A\in Y)\wedge (X\in \mathcal{V}) \wedge (Y\in \mathcal{V})$ \\ or \\ $(\forall A\in\mathcal{V})(A\in_1 X\Rightarrow A\in_1 Y)\wedge [(X\notin \mathcal{V})\vee (Y\notin\mathcal{V})]$\end{center}
\end{definition}

Remark that our extended subset definition directly reduces to the classical definition if $A,B$ are Zermelo.

Classically, if $a\in X$, $b\in X$ and $c\in X$ then we write $X$ as $\{a,b,c\}$.  Although we will not use it in the following, we now suggest a notation that helps to visualize non-Zermelo set. If $a\in_1 Y$, $b\in_1 Y$ and $c\in_1 Y$ then we write $Y$ as $\{a_1,b_1,c_1\}$. If $Y$ is non-Zermelo, then $P(Y)=\{a,b,c\}$. If $Y$ is Zermelo, this means that every subset of $a$, $b$ and $c$ is also in $Y$. For example, take $a=\{e\}, b=\{f\}$ and $c=\{e,f\}\}$ so that $Y=\{\{e\}_1,\{f\}_1,\{e,f\}_1\}$, thus it would be enough to write $Y=\{e,f\}$. By the powerset axiom, applying $P$ to $\{\{e\}_1,\{f\}_1,\{e,f\}_1\}$ gives $P(Y)=\{\{e\},\{f\},\{e,f\}\}$ and applying $P$ to $\{e,f\}$ also gives $P(Y)=\{\{e\},\{f\},\{e,f\}\}$. This notation would eventually be very useful if we decide to consider mixed sets where we could have sets such as $\{a_1, a, b, \{a,b\}_1\}$. However, for now our axioms (especially the pairing axiom) does not permit those sets. For us, the concept of cardinality will be slightly more important than the actual content of the rational sets. Cardinality will be explored in detail in the following sections.

The classical axiom of extensionality is written as follows.

 $$\forall A\forall B\forall S[S\in A \Leftrightarrow S\in B]\Rightarrow A=B$$

To be able to consider the non-Zermelo sets and their subset-members, we need to modify the axiom of extensionality.

\begin{axiom}[Extended Extensionality]$$\forall A\forall B\forall S[S\in_1 A \Leftrightarrow S\in_1 B]\Rightarrow A=B$$\end{axiom}

We want to show that the axiom of extended extensionality reduces to the axiom of extensionality of $\text{ZF}$ in the case where $A,B$ are Zermelo sets. Since $A,B$ are Zermelo sets, then by the subset assignment axiom, we have that $S\in_1 A$ and $S\in_1 B$ are respectively equivalent to $S\subseteq A$ and $S\subseteq B$. Thus for Zermelo sets, our axiom of extensionality reduces to the classical axiom of extensionality by using the following proposition.

\begin{proposition}\label{equiv}If $A,B$ are Zermelo, then $$\forall S(S\subseteq A \Rightarrow S\subseteq B)\Leftrightarrow \forall x(x\in A \Rightarrow x\in B) .$$\end{proposition}
\begin{proof}($\Rightarrow$) Take $A$ for $S$, then we have $(A\subseteq A \Rightarrow A\subseteq B)$ and this means that for all $x\in A$ we have that $x\in A\subseteq B$ and hence $x\in B$.

($\Leftarrow$) Take an arbitrary $S\subseteq A$, then for all $x\in S$ we have $x\in S\subseteq A$. By assumption, $x\in A$ implies that $x\in B$. Therefore for every $x\in S$ we must have that $x\in B$, which can be expressed as $S\subseteq B$ by our definition of subset. Hence, $\forall S(S\subseteq A \Rightarrow S\subseteq B)$.\end{proof}

\begin{proposition}\label{subsetequals0}$A\subseteq A$.\end{proposition}
\begin{proof}If $A$ is Zermelo, then it is a logical truth that $\forall x(x\in A\Rightarrow x\in A)$ and thus by the extended definition of subsets we find that $A\subseteq A$. Similarly, if $A$ is not Zermelo, then it is a logical truth that $\forall x(x\in_1 A\Rightarrow x\in_1 A)$, thus by the extended definition of subsets $A\subseteq A$.

\end{proof}

\begin{proposition}\label{subsetequals} $A\subseteq B$ and $B\subseteq A$ if and only if $A=B$.\end{proposition}
\begin{proof}($\Rightarrow$) Suppose that $A\in\mathcal{V}$ and $B\in\mathcal{V}$, then by the subset definition we have that $\forall x(x\in A\Rightarrow x\in B)$ and $\forall x(x\in B\Rightarrow x\in A)$ are true. Thus, since $\forall x(x\in A\Leftrightarrow x\in B)$, we can conclude by the extensionality axiom that $A=B$.

Otherwise, suppose that $(X\notin \mathcal{V})\vee (Y\notin\mathcal{V})$, then similarly by the subset definition we find that $\forall x(x\in_1 A\Leftrightarrow x\in_1 B)$. Therefore, by the extensionality axiom and proposition \ref{equiv}, we find that $A=B$.

($\Leftarrow$) Since $A\subseteq A$ is true and since $A=B$, we can replace the $A$ on the right hand side or the left hand side by B. Hence, we find that $A\subseteq B\wedge B\subseteq A$ is true.\end{proof}

If $A,B,C$ are all Zermelo sets, $A\subseteq B$ and $B\subseteq C$, then $A\subseteq C$. This is the classical transitivity of subsets and also follows from our definition of subsets. Since we will need a few propositions to prove the transitivity of subsets in the setting of $\text{EZF}$, we will give it at the end of the next section (see proposition \ref{transitivity}).

\subsection{Propositions}\label{propositions}

We will now prove a few essential propositions. Here is an interesting property of the Zermelo sets.

\begin{proposition}\label{zermelo prop}$$\forall Y[ Y\in\mathcal{V}\Leftrightarrow (\forall A\in\mathcal{V})[A\in_1 Y \Rightarrow (\forall S\in\mathcal{V})(S\in_1 A\Rightarrow S\in_1 Y)]]$$\end{proposition}
\begin{proof}($\Rightarrow$) Using the subset definition and proposition \ref{equiv}, $(\forall S\in\mathcal{V})(S\in_1 A\Rightarrow S\in_1 Y)$ can be written as $A\subseteq Y$. Since $Y\in \mathcal{V}$ then by the subset assignment axiom,
$A\in_1 Y \Rightarrow A\subseteq Y$ is true.

($\Leftarrow$) Suppose that $$(\forall A\in\mathcal{V})[A\in_1 Y \Rightarrow (\forall S\in\mathcal{V})(S\in_1 A\Rightarrow S\in_1 Y)]\Rightarrow Y\notin\mathcal{V}.$$ The negation of this statement gives us $$Y\in\mathcal{V}\Rightarrow \neg[(\forall A\in\mathcal{V})[A\in_1 Y \Rightarrow (\forall S\in\mathcal{V})(S\in_1 A\Rightarrow S\in_1 Y)]].$$ But again $\forall S(S\in_1 A\Rightarrow S\in_1 Y)$ can be written as $A\subseteq Y$ and thus we have $Y\in\mathcal{V}\Rightarrow \neg[(\forall A\in\mathcal{V})[A\in_1 Y\Rightarrow A\subseteq Y]]$ which contradicts the axiom of subset assignment which states that
$Y\in\mathcal{V}\Rightarrow (\forall A\in\mathcal{V})[(A\subseteq Y)\Leftrightarrow (A\in_1 Y)]$.
\end{proof}

This proposition can be seen as analogous to the characterization of the real numbers being complex numbers when they have a complex component of zero.

\begin{proposition}If $X\in\text{EZF}$, then $P(X)$ is a Zermelo set.\end{proposition}
\begin{proof}This follows immediately from the scope of the existential quantifier of the powerset axiom $\forall X(\forall A\in\mathcal{V})(\exists Y\in\mathcal{V})[A\in_1 X\Leftrightarrow A\in Y]$.\end{proof}

This proposition permits us to able to say that $P(P^{-1}(X))$ is a Zermelo set and not a non-Zermelo set. Similarly, we can prove that $P^{-1}(P(P(X)))$ is also a Zermelo set.

We now proceed to prove a few useful propositions concerning the inverse
powerset.

\begin{proposition}\label{suppesinverse}Let $X\in\mathcal{V}$ and $X\in\mathcal{V}$, then \begin{center}$X\subseteq Y$ if and only if $P^{-1}(X)\subseteq P^{-1}(Y)$.\end{center}\end{proposition}
\begin{proof}($\Rightarrow$) By assumption and by the subset definition, we have that $a\in X \Rightarrow a\in Y$. By the inverse powerset axiom, we now have that $a\in X\Leftrightarrow a\in_1 P^{-1}(X)$ and $a\in Y\Leftrightarrow a\in_1 P^{-1}(Y)$. Thus we have $a\in_1 P^{-1}(X)\Rightarrow a\in X$, $a\in X \Rightarrow a\in Y$ and $a\in Y\Rightarrow a\in_1 P^{-1}(Y)$. Therefore, by the transitivity of implication we find $a\in_1 P^{-1}(X)\Rightarrow a\in_1 P^{-1}(Y)$ which is $P^{-1}(X)\subseteq P^{-1}(Y)$ by the subset definition.

($\Leftarrow$) Similarly, by assumption and definition, we have $a\in_1 P^{-1}(X)\Rightarrow a\in_1 P^{-1}(Y)$ which can be written as $a\in X \Rightarrow a\in Y$ by the inverse powerset axioms. Thus we find $X\subseteq Y$.
\end{proof}

\begin{proposition}\label{suppes}Let $X\in\text{EZF}$ and $Y\in\text{EZF}$, then \begin{center}$X\subseteq Y$ if and only if $P(X)\subseteq P(Y)$.\end{center} \end{proposition}

\begin{proof}($\Rightarrow$) Suppose $X\in\mathcal{V}$, $Y\in\mathcal{V}$ and take $a\in P(X)$. By the inverse powerset axiom, $a\in P(X)$ is equivalent to $a\in_1 X$. By the subset assignment axiom $a\in_1 X$ is equivalent to $a\subseteq X$ and thus by the classical transitivity of subsets and our assumption $X\subseteq Y$, we find that $a\subseteq Y$. This is equivalent to $a\in_1 Y$ which is equivalent to $a\in P(Y)$. We just have shown that $a\in P(Y)\Rightarrow a\in P(Y)$ which is by definition $P(X)\subseteq P(Y)$.

Otherwise, if $X$ or $Y$ are not Zermelo, we have by definition that $a\in_1 X\Rightarrow a\in_1 Y$. By the powerset axiom this can be written as $a\in P(X)\Rightarrow a\in P(Y)$. Using the subset definition and since $P(X)$ and $P(Y)$ must be Zermelo, we also find $P(X)\subseteq P(Y)$.

($\Leftarrow$) Suppose $X\in\text{EZF}$, $Y\in\text{EZF}$ and take $a\in_1 X$, where $\in_1$ can be replaced by $\subseteq$ if $X$ is Zermelo. By the powerset axiom this is equivalent to $a\in P(X)$. By assumption, we therefore have $a\in P(X)\Rightarrow a\in P(Y)$. Since by the powerset axiom we have that $a\in P(X)\Leftrightarrow a\in_1 X$ and $a\in P(Y)\Leftrightarrow a\in_1 Y$, we find that $a\in_1 X\Rightarrow a\in_1 Y$. If $X$ and $Y$ are both Zermelo, then by the subset assignment axiom and the subset definition we find that $X\subseteq Y$. Otherwise, if $X$ and $Y$ are both not Zermelo, then by the subset definition we find $X\subseteq Y$.

\end{proof}

\begin{proposition}\label{inversesuppecor}Let $X\in\mathcal{V}$ and $Y\in\mathcal{V}$, then \begin{center}$X= Y$ if and only if $P^{-1}(X)=P^{-1}(Y)$.\end{center}\end{proposition}

\begin{proof}By proposition \ref{subsetequals}, $X=Y$ if and only if $X\subseteq Y\wedge X\subseteq Y$. Since by proposition \ref{suppesinverse}, we have that $X\subseteq Y\Leftrightarrow P^{-1}(X)\subseteq P^{-1}(Y)$, we find that $$X\subseteq Y\wedge X\subseteq Y\Leftrightarrow P^{-1}(X)\subseteq P^{-1}(Y)\wedge P^{-1}(Y)\subseteq P^{-1}(X).$$ Hence, by proposition \ref{subsetequals} we get $X= Y\Leftrightarrow P^{-1}(X)= P^{-1}(Y).$
\end{proof}

\begin{proposition}\label{suppescor}Let $X\in\text{EZF}$ and $Y\in\text{EZF}$, then \begin{center} $X= Y$ if and only if $P(X)=P(Y)$\end{center} \end{proposition}
\begin{proof}By proposition \ref{subsetequals} $X\subseteq Y \wedge Y\subseteq X\Leftrightarrow X=Y$, thus, from proposition \ref{suppes} we deduce
that $X= Y\Leftrightarrow P(X)=P(Y)$.
\end{proof}

\begin{proposition}\label{inverse}If $X\in\mathcal{V}$, then $P(P^{-1}(X))=X$.\end{proposition}

\begin{proof}
Take $x\in X$. By the inverse powerset axiom, we have that there exists a set $P^{-1}(X)$ such that $x\in_1 P^{-1}(X)$. By the powerset axiom we have that there exists a set $P(P^{-1}(X))$ such that $x\in P(P^{-1}(X))$, Therefore, we have that $x\in X\Rightarrow x\in P(P^{-1}(X))$ which means by the subset definition that $X\subseteq P(P^{-1}(X))$ since $P(P^{-1}(X))$ is Zermelo.

Take $y\in P(P^{-1}(X))$, then by the powerset axiom we have that $y\in_1 P^{-1}(X)$ and by the inverse powerset axiom we get $y\in X$. Since $P(P^{-1}(X))$ is Zermelo, we find using the subset definition $P(P^{-1}(X))\subseteq X$. By proposition \ref{subsetequals}, we can conclude that $P(P^{-1}(X))=X$.
\end{proof}

\begin{proposition}\label{inverse2}If $X\in\text{EZF}$, then $P^{-1}(P(X))=X$.\end{proposition}
\begin{proof}If $X$ is non-Zermelo, then there is a Zermelo set $X'$ such that $X=P^{-1}(X')$. Taking $P^{-1}(X')=P^{-1}(X')$ and applying $P$ on both side using proposition \ref{suppescor} we have $P(P^{-1}(X'))=P(P^{-1}(X'))$. By proposition \ref{inversesuppecor} we can write $P^{-1}(P(P^{-1}(X')))=P^{-1}(P(P^{-1}(X')))$. This becomes $P^{-1}(P(P^{-1}(X')))=P^{-1}(X')$ by proposition \ref{inverse} and since $X=P^{-1}(X')$, we find $P^{-1}(P(X))=X$.

Let $X$ be a Zermelo set and take $z\subseteq X$. Then by the powerset axiom, $z\in P(X)$ and by the inverse powerset axiom $z\in_1 P^{-1}(P(X))$. If $P^{-1}(P(X))$ is Zermelo, then $z\subseteq P^{-1}(P(X))$ and thus we have $\forall z(z\subseteq X\Rightarrow z\subseteq P^{-1}(P(X)))$. By proposition \ref{equiv}, this is equivalent to $\forall x(x\in X\Rightarrow z\in P^{-1}(P(X)))$ and therefore we have $X\subseteq P^{-1}(P(X))$. If $P^{-1}(P(X))$ is non-Zermelo then we have $z\subseteq X\Rightarrow z\in_1 P^{-1}(P(X))$ and can conclude by the subset definition that $X\subseteq P^{-1}(P(X))$.

Let $X$ be a Zermelo set and take $z\in_1 P^{-1}(P(X))$. By the powerset axiom this is equivalent to $z\in P(P^{-1}(P(X)))$ and using proposition \ref{inverse} we find $z\in P(X)$. Thus, by the powerset axiom, this becomes $z\in_1 X$. Therefore, $z\in_1 P^{-1}(P(X))\Rightarrow z\in_1 X$ and depending on whether $P^{-1}(P(X))$ is Zermelo or not, we use proposition \ref{equiv} or the subset definition to conclude that $P^{-1}(P(X))\subseteq X$.

Since $X\subseteq P^{-1}(P(X))$ and $P^{-1}(P(X))\subseteq X$, we have by proposition \ref{subsetequals}, that $P^{-1}(P(X))=X$.\end{proof}

\begin{proposition}[Uniqueness]There exists a unique set $Y$ of $\text{EZF}$ satisfying $P(Y)=X$.\end{proposition}
\begin{proof}Suppose there are two sets satisfying $P(Y)=X$, then $P(Y_1)=X=P(Y_2)$. Applying the operator $P^{-1}$ on both sides of $P(Y_1)=P(Y_2)$ we get $Y_1=Y_3$ by proposition \ref{inverse2}, a contradiction.
\end{proof}

\begin{proposition}\label{transitivity}For $A,B,C$ sets of $\text{EZF}$, If $A\subseteq B$ and $B\subseteq C$, then $A\subseteq C$. \end{proposition}

\begin{proof}If $A,B,C$ are Zermelo, we have that $A\subseteq B$ implies $\forall x(x\in A\Rightarrow x\in B)$.  Similarly,  $B\subseteq C$ implies $\forall x(x\in B\Rightarrow x\in C)$ and thus, by transitivity of implication $\forall x(x\in A\Rightarrow x\in C)$ which means that $A\subseteq C$.

If at least one set of $A,B,C$ is non-Zermelo, then we apply the powerset operator on each side of $A\subseteq B$ and $B\subseteq C$. Thus we have $P(A)\subseteq P(B)$ and $P(B)\subseteq P(C)$ and since $P(A)$, $P(B)$ and $P(C)$ are Zermelo sets, we find $P(A)\subseteq P(C)$, by classical transitivity. By proposition \ref{suppesinverse}, $P(A)\subseteq P(C)$ is equivalent to $P^{-1}(P(A))\subseteq P^{-1}(P(C))$ which is equivalent to $A\subseteq C$ by proposition \ref{inverse2}.

\end{proof}

Take $X$ to be a Zermelo set and take $Y\subseteq P(X)$ such that $Y$ is a Zermelo set and $Y\neq P(Y')$ for all Zermelo sets $Y'$, then by using proposition \ref{suppesinverse} and proposition \ref{inverse2}, we have that $P^{-1}(Y)\subseteq X$ where $P^{-1}(Y)$ is not a Zermelo set. One way to look at this result is to consider that in the $\text{EZF}$, we can take `fractions' of a set $X$. This is similar to the case of taking a fraction of an integer or like having $\frac{1}{2}\leq 2$. It is important to note that this subset is not taken into account when we construct the powerset of $X$. For $P^{-1}(Y)$ to become an element of $P(X)$ we would need to be able to write $P^{-1}(Y)\in_1 X$, but since $P^{-1}(Y)$ is non-Zermelo, this cannot be done with the subset assignment axiom.

We now have to extend the notion of cardinality. For some sets, the usual cardinality definition is adequate, for example
$|P^{-1}(P(\mathbb{N}))|= \aleph_0$ since by proposition \ref{inverse2} we have that $P^{-1}(P(\mathbb{N}))=\mathbb{N}$. For some finite sets we have a nice inverse, for example we have
$|P^{-1}(\{\emptyset,\{a\},\{b\},\{c\},\{a,b\},\{a,c\},\{b,c\},\{a,b,c\}\})|=3$ since $P^{-1}(\{\emptyset,\{a\},\{b\},\{c\},\{a,b\},\{a,c\},\{b,c\},\{a,b,c\}\})=\{a,b,c\}$. For sets such as $P^{-1}(\{1,2,3,4,5\})$ and $P^{-1}(\mathbb{N})$ it is not clear what should be its cardinality. This is what we will investigate in the following sections. In a few words, in the case of a set $X$ which cannot be written as $P(X')$ for any Zermelo set $X'$, we will have to make a choice between the two statements $|P^{-1}(X)|=|X|$ and $|P^{-1}(X)|<|X|$. In essence, the first choice allows the continuum hypothesis to be true and the second choice makes the continuum hypothesis false.

\section{Proving the Continuum Hypothesis}

By extending Zermelo-Fraenkel set theory with the concept of the inverse powerset, we now have a setting where the truth or falsity of the continuum hypothesis can be decided. This will be done by restricting $\text{EZF}$ and by giving two extensions of the definition of cardinality. The extended definitions must give the same result for Zermelo sets and must apply to non-Zermelo sets such as $P^{-1}(\{1,2,3,4,5\})$ and $P^{-1}(X)$ when $X$ cannot be written as $P(X')$ for any Zermelo set $X'$. As seen in the case of complex numbers, there are many ways to define a norm. We will give three ways to extend the definition of cardinality, in particular the first one will induce the continuum hypothesis and the other two will induce its falsity. This can also be seen as giving an explicit model of $\text{ZF}$ in which the continuum hypothesis is true and giving a model in which the continuum hypothesis is false.

In this section we will prove the continuum hypothesis in the system of $\text{EZF}$ with the classical axiom of choice, that is the axiom of choice over the Zermelo sets only. But first, here are a few definitions which will simplify our notation.

\begin{definition}A Zermelo set $X$ is said to be a powered Zermelo set if and only if there is a Zermelo set $X'$ such that $X=P(X')$. A Zermelo set which is not a powered Zermelo set will be said to be a non-powered Zermelo set.
\end{definition}

In the following, when the powerset operator appears $n$ times in the expression $P(P(P(...P(X))))$ we will denote this expression by $P^n(X)$. It is also understood that $P^0(X)=X$.

\begin{definition}A powered Zermelo set $X$ is said to be a n-powered Zermelo set if and only if $X=P^n(X')$ and $X'$ is a non-powered Zermelo set.
\end{definition}

It is important to note that for each positive integer $n$ there are many proper subsets of $P^n(\mathbb{N})$ which are non-powered Zermelo sets. In fact, there are probably many more non-powered sets than powered sets.

Since in this section we add the classical axiom of choice to $\text{EZF}$, this implies that the class of classical cardinal number is totally ordered. This way, for each Zermelo set $X$, there exist a unique integer $k\geq 1$ such that $|P^{k-1}(\mathbb{N})|<|X|\leq |P^{k}(\mathbb{N})|$.

\begin{definition}Let $X$ be a Zermelo set, then if $|P^{k-1}(\mathbb{N})|<|X|\leq |P^{k}(\mathbb{N})|$ for some integer $k\geq 1$, then we say that the Zermelo set $X$ is of degree $k$. If $|X|\leq |\mathbb{N}|$ then we say that $X$ is of degree 0.
\end{definition}

We now give an extended definition of cardinality which will induce the continuum hypothesis in the setting of $\text{EZF}$. We will use the notation $|X|$ to denote the classical cardinality of a set $X$.
\pagebreak

\begin{definition}[$CH$-cardinality]\label{chcardinalitydef}Let U be a set of $\text{EZF}$. If $U$ is a Zermelo set then
$$|U|_{ch}=|U|$$
\noindent Otherwise, if $U$ is a non-Zermelo set, let $U$ be written as $P^{-1}(X)$ with $X$ a non-powered Zermelo set of degree $k$. Then,

\noindent if $k\geq 1$, we define
\noindent $$|U|_{ch}=|P^{-1}(X)|_{ch}=|P^{k-1}(\mathbb{N})|,$$
if $k=0$, we define
$$|U|_{ch}=|P^{-1}(X)|_{ch}=|\mathbb{N}|.$$
\end{definition}

Note that if $U$ is a Zermelo set, the definitions reduce to the usual definition of cardinality. We will refer to this extension of cardinality as the \textit{CH-cardinality}. Based on the relation `$\leq$' for the classical cardinal numbers, we extend this relation to non-Zermelo sets.

\begin{definition}\label{smallch}Let U, U' be sets of $\text{EZF}$. Let $|U|_{ch}=|Y|$ and $|U|_{ch}=|Y'|$. Then,

\begin{center}$|U|_{ch}\leq |U'|_{ch}$ if and only if $|Y|\leq |Y'|$.\end{center}
\end{definition}

We now prove the extension of the Schroeder-Bernstein theorem.

\begin{proposition}\label{Bernstein}If $|U|_{ch}\leq |U'|_{ch}$ and $|U'|_{ch}\leq |U|_{ch}$ then $|U|_{ch}=|U'|_{ch}$.\end{proposition}
\begin{proof}Let $|U|_{ch}=|Y|$ and $|U|_{ch}=|Y'|$, then by definition \ref{smallch} we have that $|Y|\leq |Y'|$ and $|Y'|\leq |Y|$ which implies by the Shroeder-Bernstein theorem that $|Y|=|Y'|$. Hence, we have $|U|_{ch}=|U'|_{ch}$.
\end{proof}

\begin{proposition}If $|U|_{ch}\leq |V|_{ch}$ and $|V|_{ch} \leq |W|_{ch}$ then $|U|_{ch}\leq |W|_{ch}$.\end{proposition}
\begin{proof}Let $|U|_{ch}=|Y|$, $|V|_{ch}=|Y'|$ and $|W|_{ch}=|Y''|$, then by assumption and by the classical transitivity, we have $|Y|\leq |Y''|$ and thus, $|U|_{ch}\leq |W|_{ch}$.\end{proof}

Following the approach given in \cite{suppes}, we extend a few useful definitions.

\begin{definition}\label{chcardinal}$\alpha$ is a CH-cardinal number if and only if there is a set $X$ of $\text{EZF}$ such that $|X|_{ch}=\alpha$.\end{definition}

In \cite{suppes}, there is a similar definition, but for cardinal numbers in $\text{ZF}$. The validity of that definition is assured by the axiom of cardinality. For more detail, see p.111 of \cite{suppes}. Thus, in $\text{EZF}$, the validity of definition \ref{chcardinal} is guaranteed by the axiom of cardinality of \cite{suppes} and the CH-cardinality definition.

\begin{proposition}\label{cardsubsetofch}If $\alpha$ is a cardinal number, then $\alpha$ is a CH-cardinal number.\end{proposition}
\begin{proof}If $\alpha$ is a cardinal number, by the usual definition of cardinality there is a set of $\mathcal{V}$ such that $|X|=\alpha$. By the CH-cardinality definition \ref{chcardinalitydef}, we have $|X|_{ch}=|X|=\alpha$ since $X$ is Zermelo, thus by definition \ref{chcardinal}, $\alpha$ is a CH-cardinal number.\end{proof}

In \cite{suppes}, Suppes give the following two definitions regarding the context of $\text{ZF}$. To help in comparing the notations, in the following, we will consider $\text{ZF}$ to be the Von Newmann universe $\mathcal{V}$.

\begin{definition}\label{a0}Let $A,B$ be sets of $\text{ZF}$, then $A\preceq B$ if and only if there exists a set $C$ of $\text{ZF}$ such that $|C|=|A|$ and $C\subseteq B$.\end{definition}

\begin{definition}\label{b0} $\alpha \leq \alpha'$ if and only if there are sets $A$ and $B$ of $\text{ZF}$ such that $|A|=\alpha$, $|B|=\alpha'$ and $A\preceq B$.\end{definition}

We extend those definitions to $\text{EZF}$ in the following way.

\begin{definition}\label{a}Let $A,B$ be sets of $\text{EZF}$, then $A\preceq B$ if and only if there is a set $C$ of $\text{EZF}$ such that $|C|_{ch}=|A|_{ch}$ and $C\subseteq B$.\end{definition}

\begin{definition}\label{b} $\alpha \leq \alpha'$ if and only if there are sets $A$ and $B$ of $\text{EZF}$ such that $|A|_{ch}=\alpha$, $|B|_{ch}=\alpha'$ and $A\preceq B$.\end{definition}

Note that if we take $A,B$ to be Zermelo sets, the definitions \ref{a} and \ref{b} reduce to \ref{a0} and \ref{b0}, since by the CH-cardinality definition $|A|_{ch}=|A|$ and $|B|_{ch}=|B|$.

\begin{lemma}\label{propA} Let $A$ be a set of $\text{EZF}$ and let $B$ be a Zermelo set, then $|A|_{ch}\leq|B|_{ch}$ if and only if there is a Zermelo set $C''$ such that $|C''|_{ch}=|A|_{ch}$ and $C''\subseteq B$.\end{lemma}
\begin{proof}($\Rightarrow$)Using definitions \ref{b} and \ref{a},  there is a set $C$ of $\text{EZF}$ such that $|C|_{ch}=|A|_{ch}$ and $C\subseteq B$.  By the definition of CH-cardinality, $|C|_{ch}=|C'|$ for some Zermelo set $C'$ and thus $|C|_{ch}=|C'|_{ch}$. Since $B$ is Zermelo we have $|B|_{ch}=|B|$ and since $|C|_{ch}=|A|_{ch}$ and $|C|_{ch}=|C'|_{ch}=|C'|$, we have that $|A|_{ch}=|C'|$. By assumption, $|A|_{ch}\leq|B|_{ch}$, hence by replacing $|A|_{ch}$ and $|B|_{ch}$, we find $|C'|\leq|B|$. By definition \ref{b0} and \ref{a0},  there exists a set $C''$ of $\text{ZF}$ such that $|C''|=|C'|$ and $C''\subseteq B$. This is what we were looking for since $|C''|_{ch}=|C''|=|C'|=|C|_{ch}=|A|_{ch}$.

($\Leftarrow$) By definitions \ref{a} and \ref{b} and since Zermelo sets are sets of $\text{EZF}$ we find that $|A|_{ch}\leq|B|_{ch}$.
\end{proof}

\begin{definition}\label{c} $\alpha<\alpha'$ if and only if $\alpha\leq\alpha'$ and not $\alpha'\leq\alpha$.\end{definition}

\begin{definition} Let $U$ be a set of $\text{EZF}$ and let $|U|_{ch}=|Y|$, then $U$  is finite if and only if $|Y|$ is finite. A set of $\text{EZF}$ is infinite if it is not finite.\end{definition}

We will now evaluate the CH-cardinality of $P^{-1}(\mathbb{N})$. It is important to note that by definition, $\mathbb{N}$ is a Zermelo set, but it is not clear if it is a powered Zermelo set or not.

\begin{theorem}\label{minimality}Let $A$ be a Zermelo set such that $|A|_{ch}=|\mathbb{N}|_{ch}$ then $|P^{-1}(A)|_{ch}=|\mathbb{N}|_{ch}$.\end{theorem}
\begin{proof}If $A$ is not a powered Zermelo set, by definition of CH-cardinality, we have $|P^{-1}(A)|_{ch}=|A|$. Since $|\mathbb{N}|=|\mathbb{N}|_{ch}$ and by assumption $|A|_{ch}=|\mathbb{N}|_{ch}$, we find that $|P^{-1}(A)|_{ch}=|\mathbb{N}|_{ch}$. Thus, if we show that $A$ cannot be a powered Zermelo set, we are finished.

Assume that $A$ is a powered Zermelo set, then there is a Zermelo set $A'$ (powered or not) such that $P(A')=A$.

Suppose that $|P^{-1}(A)|_{ch}\geq|\mathbb{N}|_{ch}$, then replacing $A$ we find $|P^{-1}(P(A'))|_{ch}\geq|\mathbb{N}|_{ch}$ which becomes $|A'|_{ch}\geq|\mathbb{N}|_{ch}$ by proposition \ref{inverse2}. Since $|A'|=|A'|_{ch}\geq|\mathbb{N}|_{ch}=|\mathbb{N}|$, we find $|P(A')|\geq |P(\mathbb{N})|$ by using the well known result $|X|\geq|Y|\Rightarrow |P(X)|\geq|P(Y)|$ (see the lemma on p.95 of \cite{jech}). But since $P(A')=A$ this means that $|\mathbb{N}|=|A|=|P(A')|_{ch}\geq |P(\mathbb{N})|$, a contradiction to Cantor's theorem.

Now, suppose that $|P^{-1}(A)|_{ch}<|\mathbb{N}|_{ch}$. Since $A$ is a powered Zermelo set, then there is a Zermelo set $A'$ such that $P(A')=A$. Replacing on the left hand side of  $|P^{-1}(A)|_{ch}<|\mathbb{N}|_{ch}$ we find $|P^{-1}(P(A'))|_{ch}<|\mathbb{N}|_{ch}$ and thus by proposition \ref{inverse2} and because $A'$ and $\mathbb{N}$ are Zermelo sets, $|A'|=|A'|_{ch}<|\mathbb{N}|_{ch}=|\mathbb{N}|$. Using the definition of \cite{jech} p.74, which says that a subset $S$ is at most countable if $|S|\leq |\mathbb{N}|$, this mean that $A'$ is at most countable. By the corollary of p.74 of \cite{jech}, a set is at most countable if and only if it is finite or countable. If $A'$ is finite, we find that $P(A')=A$ is finite, a contradiction with $|A|=|\mathbb{N}|$. Thus, we must have that $A'$ is countable. This means that $|A'|=\aleph_0$. Thus we have that $\aleph_0=|A'|=|A'|_{ch}<|\mathbb{N}|_{ch}=|\mathbb{N}|=\aleph_0$, an impossibility.
\end{proof}

We are now in a context where we can prove the continuum hypothesis.

\begin{theorem}In $\text{EZF}$, there exists no CH-cardinal number $\beta$ such that $\aleph_0 < \beta < \mathfrak{c}$ where $\aleph_0$ is the CH-cardinality of a countable set and $\mathfrak{c}$ is the CH-cardinality of the continuum. \end{theorem}
\begin{proof}Assume that there is such a CH-cardinal number $\beta$. Since $\aleph_0=|\mathbb{N}|_{ch}=|\mathbb{N}|$ and $\mathfrak{c}=|P(\mathbb{N})|_{ch}=|P(\mathbb{N})|$ we have by definition \ref{chcardinal}, that there exists set $B$ of $\text{EZF}$ such that $|\mathbb{N}|_{ch}< |B|_{ch} < |P(\mathbb{N})|_{ch}$.

Since $|B|_{ch}< |P(\mathbb{N})|_{ch}$, by lemma \ref{propA} and definition \ref{c}, there exists a Zermelo set $B'$ such that $B'\subset P(\mathbb{N})$,  $|B'|_{ch}=|B|_{ch}$ and $|B'|_{ch}\neq|P(\mathbb{N})|_{ch}$. Also, since $|\mathbb{N}|_{ch}< |B|_{ch}=|B'|_{ch}$, by lemma \ref{propA}, there exists a Zermelo set $N'$ such that $N'\subset B'$,  $|N'|_{ch}=|\mathbb{N}|_{ch}$ and $|N'|_{ch}\neq|B'|_{ch}$.

Hence, by the proposition \ref{transitivity} of transitivity, we have $N'\subset B'\subset P(\mathbb{N})$. By proposition \ref{suppesinverse}, we find $P^{-1}(N')\subset P^{-1}(B')\subset P^{-1}(P(\mathbb{N}))$ and by proposition \ref{inverse}, we get $P^{-1}(N')\subset P^{-1}(B')\subset \mathbb{N}$. By definition \ref{a} and \ref{b} this means that $|P^{-1}(N')|_{ch}\leq |P^{-1}(B')|_{ch}\leq |\mathbb{N}|_{ch}$. By theorem \ref{minimality}, we have $|\mathbb{N}|_{ch}=|P^{-1}(N')|_{ch}\leq |P^{-1}(B')|_{ch}\leq |\mathbb{N}|_{ch}$ and by proposition \ref{Bernstein}, we find that $|P^{-1}(B')|_{ch}=|\mathbb{N}|_{ch}$.

We now have two cases to consider: $B'$ is a powered Zermelo set or $B'$ is not a powered Zermelo set.

If $B'$ is a powered Zermelo set, then we can write $B'$ as $P^n(X')$ with $X'$ a Zermelo set. Since the powerset of a Zermelo set is Zermelo, we can write $B'$ as $P(X)$ with $X$ being a Zermelo set. Replacing $B'=P(X)$ in $|\mathbb{N}|_{ch}=|P^{-1}(B')|_{ch}$, we find $|\mathbb{N}|_{ch}=|P^{-1}(P(X))|_{ch}$ which becomes by proposition \ref{inverse}, $|\mathbb{N}|_{ch}=|X|_{ch}$. Since $\mathbb{N}$ and $X$ are Zermelo we find, by the classical result $|X|=|Y|\Rightarrow |P(X)|=|P(Y)|$ (see p.95 lemma of \cite{jech}), that  $|P(\mathbb{N})|_{ch}=|P(X)|_{ch}$. But since $B'=P(X)$ and $|B'|_{ch}=|B|_{ch}$ we have $|P(\mathbb{N})|_{ch}=|B|_{ch}$ which is a contradiction with our primary assumption that $|P(\mathbb{N})|_{ch}>|B|_{ch}$.

Let $B'$ be a non-powered Zermelo set, then we have that $P^{-1}(B')$ is a non-Zermelo set. The CH-cardinality definition then tells us that $|P^{-1}(B')|_{ch}=|B'|_{ch}$. Thus, since we have found that $|P^{-1}(B')|_{ch}=|\mathbb{N}|_{ch}$ and since $|B'|_{ch}=|B|_{ch}$, we have that $|B|_{ch}=|B'|_{ch}=|P^{-1}(B')|_{ch}=|\mathbb{N}|_{ch}$, a contradiction with our primary assumption that $|\mathbb{N}|_{ch}<|B|_{ch}$.

\end{proof}

\begin{corollary}In $\text{EZF}$, there exists no cardinal number $\beta$ such that $\aleph_0 < \beta < c$.\end{corollary}
\begin{proof}This follows from \ref{cardsubsetofch}, since every cardinal number is a CH-cardinal number.\end{proof}

\section{Transfinite Empty Sets}\label{emptysets}

This section aims at discussing an interesting topic involving the concept of inverse powerset while not formally establishing the theory. Cantor showed that there are infinite sets of different cardinalities (see Cantor in \cite{Ewald}). In this section, we suggest an investigation of the same idea but for empty sets. For infinite sets, it was clear from the beginning that there were many distinct infinite sets. In $\text{ZF}$, it was difficult to consider the existence of more than one empty set, in particular what could be more empty than an object which does not contain an uncountable number of objets? In $\text{EZF}$, we are now in a position to define another empty set. The goal of this section is to suggest systems where we would have more than one empty set of different cardinalities.

Recall that the empty set is defined as follows, where $X$ is will be denoted as $\emptyset_0$.

$$(\exists X\in\mathcal{V}) (\forall y\in\mathcal{V})\neg(y\in X)$$

Since we extended Zermelo set theory and defined the collection of sets $\text{EZF}$, we can now consider another type of empty set with the following statement.

$$(\exists Z\in\mathcal{V}) (\forall y\in\text{EZF})\neg(y\in Z)$$

We will denote $Z$ as $\emptyset_1$.

The question is to ask if those two `empty sets' have the same cardinality. At first glance, an empty set is a set devoid of elements and it is unique, but we can look more closely and consider which elements are not in the set. From this perspective, there are more elements which are not in $\emptyset_1$. The classical definition of set difference tells us that if $A\subset  B$, then $A\setminus A= A\setminus B=\emptyset$. Based on the idea of subtraction of integers, it would be interesting to consider $A\setminus A$ as having a higher cardinality than $A\setminus B$ by considering $A\setminus B$ to have `negative' elements.

One avenue is to add new axioms which define new sets $x^-, y^-, z^-$ such that $\{x,x^-\}=\emptyset$. For example, we could have that $\{a,b\}\setminus \{a,b,c\}=\{c^-\}$. Although we will not do this in the present text, it could involve axioms of the following type. In the following statements consider that the superscript `$-$' is a constant of the language.

$$\forall A\exists B[x\in A\Rightarrow x^- \in B]$$
$$\forall x[\{x,x^-\}=\emptyset]$$
$$\forall Z\forall X\exists Y[X\cup Y= Z]$$
$$\forall D\forall X\forall d[(d\in D\wedge d\notin X) \Rightarrow (d^-\in X\setminus (X\cup D))]$$

Notice that the third suggestion seems to induce the `algebraic closure' for the union.

Another avenue is to consider the following definition.

\begin{definition}Two equipotent sets $X,Y$ are said to be strongly equipotent if and only if there exists a bijective function between the set of elements which are not in $X$ and the set of elements which are not in $Y$.\end{definition}

In most classical considerations, in particular in Zermelo sets defined without using the complement and the set difference and formulas of the type `$x\notin A$, we can say that equipotent sets are strongly equipotent. Classically, when we consider the set $X=\{a,b,c\}$ we implicitly think of this set as the set which does not have all Zermelo sets except $a,b$ and $c$. Thus, if $X=\{a,b,c\}$ and $Y=\{d,e,f\}$ then $X$ and $Y$ are strongly equipotent. But it is not the case for $A\setminus A$ and $A\setminus B$ when $A\subset B$, since there are elements of $B$ which are not in $A$.

Coming back to our two empty sets $\emptyset_0$ and $\emptyset_1$, we could show that $\emptyset_0$ and $\emptyset_1$ are empty sets which are not strongly equipotent. In essence, we would need to show that there is no bijection between $\mathcal{V}$ and $\text{EZF}$. Notice that there seems to be many more sets in $\text{EZF}$ than in $\mathcal{V}$.

The idea behind a proof is as follows. We want to show that $|\text{EZF}|=|P(\mathcal{V})|$, so that we have $|\mathcal{V}|<|\text{EZF}|$.  Each proper subset $S$ of the class $\mathcal{V}$ is a set of $\mathcal{V}$. For all such proper subsets $S$, the set $P^{-1}(\mathcal{V}\setminus S)$ is a set of $\text{EZF}$. If we show that each distinct $S$ corresponds to a unique element of $\text{EZF}$, we would find under this correspondence that $|\text{EZF}|=|P(\mathcal{V})|$ which means that $|\mathcal{V}|<|\text{EZF}|$.

Hence, by the definition of strongly equipotent, since $|\mathcal{V}|<|\text{EZF}|$ we have that $\emptyset_1$ and $\emptyset_0$ are not strongly equipotent. Thus after defining the concept of `strong cardinality' with its associated order relation we would have that $|\emptyset_1|_s<|\emptyset_0|_s$. Here the fact that $|\mathcal{V}|<|\text{EZF}|$ allows us to conclude $|\emptyset_1|_s<|\emptyset_0|_s$, but it is important to note this would not be so if you extended $\mathcal{V}$ with a finite number of sets or even with an infinite number of sets equal to or smaller than $\mathcal{V}$. In fact, in our approach under the strongly equipotent definition, there are many more empty sets smaller than $\emptyset_0$, for example $A\setminus C$ where $C=A\cup\{b\}$ and $b\notin A$.

We will see in the next section, that after defining the $N$th level inverse powerset axioms we will have the opportunity to consider the following sequence of empty sets.
$$|\emptyset_0|_s>|\emptyset_1|_s>|\emptyset_2|_s>...>|\emptyset_n|_s>...$$

Note that, for example, the definition of $\emptyset_2$ relies on the extension of $\text{EZF}$ with objects of the form $P^{-1}(P^{-1}(X))$. We might also be able to eventually consider a sequence where $n$ is any ordinal number. It is possible that a formal proof of $|\mathcal{V}|<|\text{EZF}|$ needs the transfinite induction.

\section{Nth Level Inverse Powerset Axioms}\label{level n}

The powerset operator can be applied repeatedly to a set, for example $P(P(P(X)))$. We would also like to do this with the inverse powerset to have sets of the form $P^{-1}(P^{-1}(P^{-1}(X)))$. The propositions seen in section \ref{EZF level 1} as well as the continuum hypothesis seem to generalize well to a theory including such objects, but note that for our purpose we will not need to prove them since the proof of the falsity of the continuum hypothesis does not rely on them.

To the system of axioms of $\text{EZF}$, we add the following axioms for each integer $n\geq 1$.

\begin{axiom}[Powerset Schema]$$\forall X\exists Y\forall A[A\in_{n+1} X\Leftrightarrow A\in_{n} Y]$$\end{axiom}

We will denote $Y$ as $P(X)$ and call it the \textit{powerset} of $Y$.

\begin{axiom}[Inverse powerset Schema]$$\forall Y\exists X\forall A[A\in_{n} Y\Leftrightarrow A\in_{n+1} X]$$\end{axiom}

We will denote $X$ as $P^{-1}(Y)$ and call it the \textit{inverse powerset} of $Y$.

\begin{axiom}[Extended Extensionality]$$\forall A\forall B\forall S[S\in_n A \Leftrightarrow S\in_n B]\Rightarrow A=B$$\end{axiom}

We also modify our definition of subsets.

\begin{definition}[Extended Subsets]\label{extsubset} $A\subseteq B$ if and only if
\medskip
\begin{center}$(\forall A\in\mathcal{V})(A\in X\Rightarrow A\in Y)\wedge (X\in \mathcal{V}) \wedge (B\in \mathcal{V})$ \\ or, for some $n$ \\ $\forall A(A\in_n X\Rightarrow A\in_n Y)\wedge [(X\notin \mathcal{V})\vee (Y\notin\mathcal{V})]$\end{center}
\end{definition}

We now give a useful definition which will help our notation.

\begin{definition}If $P^{-1}$ occurs $n$ times in $P^{-1}(P^{-1}(...P^{-1}(X)...))$ and $X$ is a Zermelo set such that $X\neq P(X')$ for all Zermelo sets $X'$, we will denote $P^{-1}(P^{-1}(...P^{-1}(X)...))$ as $P^{-n}(X)$ and we will say that it is a non-Zermelo set of level $n$.
\end{definition}

We know that for Zermelo sets, the axiom of subset assignment tells us that $x\subseteq X$ is equivalent to $x\in_1 X$. It is important to note that by construction there are no sets except for the Zermelo sets which have two different type of members. For example, a set cannot be such that $A\in_2 X$ and $B\in X$. As in $\text{EZF}$, the axiom of pairing and of union are defined over the Zermelo sets only, hence for now there are no sets which have `mixed' elements. In this system, if an element $z\in_2 Z$, then we can assume that Z is a set of the form $P^{-2}(Z'')$ for some Zermelo set $Z''$. Here, the axiom of inverse powerset schema allows us to construct sets of the form $P^{-k}(W)$ and the axiom powerset schema allows us to apply the powerset operator to such an object to get $P(P^{-k}(W))$.

We now define the class $\text{EZF}^p$ where $\text{EZF}^p$ contains all Zermelo sets and all non-Zermelo sets of any level. Note that by definition, the operators $P$ and $P^{-1}$ are applied a countable infinite number of times. The collection $\text{EZF}^p$ can also be viewed as $\text{ZF}$ augmented with non-Zermelo sets of all levels.

\pagebreak

\begin{definition}\label{terms}We define the collection $\text{EZF}^p$ of real sets from $\text{EZF}$ as:

\begin{itemize}
\item[1.] If $U$ is a Zermelo set, then $U\in \text{EZF}^p$,
\item[2.] If $U$ is a Zermelo set and $n$ is a positive integer, then $P^{-n}(U)\in \text{EZF}^p$,
\end{itemize}
\end{definition}

By definition, we have that a set of $\text{EZF}^p$ is a Zermelo set or a non-Zermelo set of a certain level. Note that this definition could also be generalized further to ordinal numbers by allowing an uncountable number of occurrences of the operators of powerset and inverse powerset, but would need transfinite recursion.

\begin{definition}We will say that two sets $X$, $Y$ are of the same level if and only if $m>0$, $X=P^{-m}(X')$ and $Y=P^{-m}(Y')$ for some non-powered Zermelo sets $X'$ and $Y'$ or if they are both non-powered Zermelo sets.\end{definition}

\section{Disproving the Continuum Hypothesis}\label{section falsity of CH}

We now give two extended definitions of cardinality which will induce the falsity of the continuum hypothesis in the context of $\text{EZF}^*$ which will be defined below. The first extended definition is much weaker than the second in the sense that in the second there are many more $\neg$CH-cardinal numbers between $\mathbb{N}$ and $\mathfrak{c}$. The definitions are closely related to the lexicographic order. We will call the weaker extended definition of cardinality \textit{$\neg$CH-cardinality} and the other stronger definition \textit{$\neg$CHS-cardinality}. We will see that those definitions give a richer theory of cardinality, which is similar to enriching the natural numbers with the rational numbers or real numbers.

In essence, the definitions and axioms of this section and of section \ref{level n} only aim at formalizing objects of the type $X\cup P^{-1}(Y_1)\cup ...\cup P^{-1}(Y_{n_1})\cup ...\cup P^{-k}(Z_1)\cup ...\cup P^{-k}(Z_{n_k})$. For the developments associated to the falsity of the continuum hypothesis, it would be enough to construct a class containing those objects such that the internal structure is ignored.

For the following, we now add to our system $\text{EZF}^p$ an axiom which we will call paired forms where the variables under the quantifiers are not only restricted to Zermelo sets but can take any element of $\text{EZF}^p$. The only use we will make of this axiom is to construct the union forms.

\begin{axiom}(Paired Forms) Let $C$ be called a paired form and let $\mathcal{P}$ be the class of paired forms,
\begin{center}$(\forall A\in\text{EZF}^p)(\forall B\in\text{EZF}^p)(\exists C\in\mathcal{P})\forall X(X\in C\Leftrightarrow (X= A\vee X=B))$.\end{center}\end{axiom}

We will say that $X\in\mathcal{P}$ is a \textit{paired form} and note that we do not consider it to be a set, this means that a paired form cannot appear in any other axiom (except the axiom of union forms which will be defined below). Here are some examples $\{1,2,X, Y, P^{-1}(Z), P^{-6}(1,3,A,B)\}$ and $\{1,2, 4, P^{-1}(Y), P^{-1}(Z)\}$. It is interesting to note that most of the results presented above can be proved by using paired forms instead of Zermelo sets, but we will not consider this further.

For our purpose, we need to extend the axiom of union. We take $\mathcal{P}$ to be the class of paired form, where a paired form is a collection of sets of $\text{EZF}$. For each $n$ we have the following axiom.

\begin{axiom}[Union Forms] Let $D$ be called a union form and let `$\in_0$' be the usual membership `$\in$'. For all paired forms $C\in\mathcal{P}$ and for all $n\geq 0$,
\begin{center}$(\exists D\in\mathcal{U})\forall S[S\in_n D\Leftrightarrow \exists B(S\in_n B \wedge B\in C)].$\end{center}\end{axiom}

If $C$ has a finite number of elements, then we will denote $D$ as $A_1\cup A_2\cup ... \cup A_n$ if $C=\{A_1, A_2, ... , A_n\}$ and say that $D$ is a \textit{finite union form}. Again, a union form is not a set or a paired form so the other axioms cannot apply to them. Note that the union is commutative and associative since $C$ is not ordered.

If $C,B$ and $S$ are Zermelo, the union forms axiom is equivalent to the axiom of union of ZF, except for the fact that the new objects are members of the class $\mathcal{U}$ and not the universe of discourse of ZF.

The axiom of extended union gives rise to a wider range of objects, for example $P^{-1}(\{1,2,3,4,5\})\cup P^{-1}(P^{-1}(X))$ and $P^{-1}(X) \cup P^{-1}(Y)\cup Z \cup P^{-1}(P^{-1}(W))$.

We have union forms which can have an infinite number of operators $P$ and $P^{-1}$, hence for the following, we will restrict the domain. To define $\text{EZF}^p$, we will take the Zermelo sets, the non-Zermelo sets of arbitrary level and the sets which arise from the finite union of a combination of Zermelo. For example, $X\cup Y \cup P^4(Z)\cup P^{-3}(U) \cup P^{-3}(V) \cup P^{-4}(W)$ is a set of $\text{EZF}^p$. Since the classical union is commutative and associative and since the union of two Zermelo sets is a Zermelo set, we can combine all the Zermelo sets together. Also, since $P^{-1}(P(X))=X$, we have that every set of $\text{EZF}^*$ can be written as
$$X\cup P^{-1}(Y_1)\cup P^{-1}(Y_2)\cup ...\cup P^{-1}(Y_{n_1})\cup ...\cup P^{-k}(Z_1)\cup P^{-k}(Z_2)\cup ...\cup P^{-k}(Z_{n_k}),$$
with $X,Y_1,Y_2,...,Y_{n_1},...,Z_1,Z_2,...,Z_{n_k}$ are all non-powered Zermelo sets. Formally, we define $\text{EZF}^*$ as follows.

\begin{definition}\label{terms}We define the collection $\text{EZF}^*$ of real sets from $\text{EZF}$ with the extended axiom of union as:

\begin{itemize}
\item[1.] If $X$ is a Zermelo set of $\mathcal{V}$, then $X\in \text{EZF}^*$,

\item[2.] If $X$ is a non-Zermelo set of level $k$ of $\text{EZF}^p$, then $X\in \text{EZF}^*$,

\item[3.] If $Z$ is a Zermelo finite union form of $\mathcal{U}$, then $X\in \text{EZF}^*$.
\end{itemize}
\end{definition}

We define two functions which will be used in the extended definitions of cardinality. In the following, we will not consider that we can write every real set similar to $P^{-1}(A)\cup P^{-1}(B)$ into a real set $P^{-1}(C)$ for some Zermelo set $C$. The reason we have introduced the function $\tau$ is that we want to stay in a more general setting since the concept of union presented here might not be the only interesting approach.

\begin{definition}Let $Y_1$ be Zermelo and let $Y_2$ be a non-empty non-powered Zermelo set and let $m$ be a positive integer. We define $\rho$, a function which takes as input a set $X$ of $\text{EZF}^p$ and returns an integer or $-\infty$, as follows:

$$\rho(X) = \left\{
     \begin{array}{ll}
       0 & \text{if $X=Y_1$} \\
       -m & \text{if $X=P^{-m}(Y_3)$}\\
       -\infty & \text{if $X=\emptyset$}
     \end{array}\right.$$
\end{definition}

We now give a definition which returns the level of non-Zermelo sets of $\text{EZF}^p$.

\begin{definition}Let $Y_1$ be Zermelo and let $Y_2$ be a non-powered Zermelo set and let $m$ be a positive integer. We define $\tau$, a function which takes as input a set $X$ of $\text{EZF}^p$ and returns a cardinal number, as follows:

$$\tau(X) = \left\{
     \begin{array}{ll}
       |Y_1| & \text{if $X=Y_1$} \\
       |Y_3| & \text{if $X=P^{-m}(Y_2)$}
     \end{array}\right.$$
\end{definition}

We now give a useful formal definition, but in a few words, it is equivalent to ordering the components of the union in the following way: $$X\cup P^{-1}(Y_1)\cup ...\cup P^{-1}(Y_{n_1})\cup ...\cup P^{-k}(Z_1)\cup ...\cup P^{-k}(Z_{n_k})$$
\noindent such that $|Y_1|\leq |Y_2|\leq ...\leq |Y_{n_1}|$ and ... and $|Z_1|\leq ...\leq |Z_{n_k}|$.

We now give a definition which returns the non-powered Zermelo set to which the inverse powerset has been applied.

\begin{definition}Let $X$ be a set of $\text{EZF}^*$ such that $X=X_1\cup X_2 \cup ... \cup X_n$, then
$X$ is well-represented if and only if \begin{center}$\rho(X_1)\geq ... \geq \rho(X_n)$

and

if $\rho(X_i)= ... = \rho(X_{i+k})$ then $\tau(X_i)\geq ... \geq \tau(X_{i+k})$ .\end{center}\end{definition}

\subsection{$\neg$Continuum Hypothesis}

In analogy with numbers, the definition of the weaker extended definition of cardinality \textit{$\neg$CH-cardinality} can be seen as ordering numbers in the following manner.

$$0.223<0.224<0.225<...<0.255<...$$$$<0.324<0.3241<0.3242<1<...$$$$<1.111<...<1.99999<...$$

\begin{definition}[$\neg CH$-cardinality]\label{notcontinuum}Let X, Y be sets of $\text{EZF}^*$. Let $X=X_1\cup X_2 \cup ... \cup X_n$ and let $Y=Y_1\cup Y_2 \cup ... \cup Y_m$ where only $X_1,Y_1$ are Zermelo (possibly empty) and where $X_2,...,X_n, Y_2,...,Y_m$ are all non-empty and non-Zermelo such that $X_2\cup X_3 \cup ... \cup X_n$ and $Y_2\cup Y_3 \cup ... \cup Y_m$ are both well-represented. Then,

$|X|_{\neg ch}\leq |Y|_{\neg ch}$ if and only if  $|X_1|<|Y_1|$ or each of the following are simultaneously valid:
\begin{center}$\begin{array}{l}
    |X_1|=|Y_1|, \\
    n\leq m, \\
	\text{$\rho(X_i)\leq\rho(Y_i)$ for all $i$ such that $n\geq i\geq 2$}, \\
	\text{$\tau(X_i)\leq\tau(Y_i)$ for all $i$ such that $n\geq i\geq 2$}, \\
		\end{array}$\end{center}

\noindent We have equality on the left hand side if and only if $|X_1|=|Y_1|$, $n=m$, $\rho(X_i)=\rho(Y_i)$ and $\tau(X_i)=\tau(Y_i)$.
\end{definition}

By construction of the axiom of paired forms and union forms, we have that our extended union is commutative and associative. Since the union is commutative and since the union of two Zermelo sets is a Zermelo set, each set of $\text{EZF}^*$ can be well-represented. Note that the transitivity of $\leq$ on the $\neg CH$-cardinal numbers follows from the transitivity of $\leq$ on the integers and $\leq$ on the classical cardinal numbers.

\begin{proposition}Let $X$ and $Y$ be Zermelo, then $|X|_{\neg ch}\leq |Y|_{\neg ch}$ if and only if $|X|\leq|Y|$.\end{proposition}
\begin{proof}Since $X$ and $Y$ are Zermelo,  we have that $X=X_1$, $Y=Y_1$ and $m=n=1$. By definition of $\neg CH$-cardinality we have that  $|X|_{\neg ch}\leq |Y|_{\neg ch}$ if and only if  [$|X_1|<|Y_1|$ or ($|X_1|=|Y_1|$ and $m=n$)]. But $|X_1|<|Y_1|$ or $|X_1|=|Y_1|$ is equivalent to $|X|\leq|Y|$ and hence we find $|X|_{\neg ch}\leq |Y|_{\neg ch}\Leftrightarrow|X|\leq|Y|$.\end{proof}

We are now in a context where we can prove the falsity of the generalized continuum hypothesis.

\begin{theorem}[$\neg$ Continuum Hypothesis]Let $Z$ be a Zermelo set such that $|Z|_{\neg ch}<|P(Z')|_{\neg ch}$ where $Z'$ is a non-powered Zermelo set, then $|Z|_{\neg ch}<|Z\cup P^{-n}(Z')|_{\neg ch}$ and $|Z\cup P^{-n}(Z')|_{\neg ch}<|P(Z)|_{\neg ch}$ where $n$ is any integer greater than or equal to $1$.\end{theorem}

\begin{proof}First, we must show that $|Z|_{\neg ch}<|Z\cup P^{-n}(Z')|_{\neg ch}$. In the definition of $\neg CH$-cardinality, let $X=Z$ and $Y=Z\cup P^{-n}(Z')$, then since $|X_1|=|Z|=|Y_1|$, $n=1 < 2=m$, $\rho(X_1)=\rho(Y_1)$ and $\tau(X_1)=\tau(Y_1)$,  we have by definition \ref{notcontinuum} that $|Z|_{\neg ch}<|Z\cup P^{-n}(Z')|_{\neg ch}$.

Secondly, we must show that $|Z\cup P^{-n}(Z')|_{\neg ch}<|P(Z)|_{\neg ch}$. Taking $X=Z\cup P^{-n}(Z')$ and $Y=P(Z)$, then since $|Z|=|X_1|<|Y_1|=|P(Z)|$ by Cantor's theorem, we have by definition of $\neg CH$-cardinality, $|Z\cup P^{-n}(Z')|_{\neg ch}<|P(Z)|_{\neg ch}$.
\end{proof}

Since $n$ is an arbitrary integer, the statement of the theorem tells us that there are infinitely many $\neg CH$-cardinal numbers between $|\mathbb{N}|_{\neg ch}$ and $|P(\mathbb{N})|_{\neg ch}$. Moreover, strictly between $|Z\cup P^{-n}(Z)|_{\neg ch}$ and $|P(Z)|_{\neg ch}$ there are also infinitely many $\neg CHW$-cardinal numbers. Taking $n\geq 2$, examples are: \begin{center}$|Z\cup P^{-n+1}(Z')|_{\neg ch}$, $|Z\cup P^{-n}(Z')\cup P^{-n-1}(Z')|_{\neg ch}$, $|Z\cup P^{-n}(Z')\cup P^{-n-2}(Z')|_{\neg ch}$, $|Z\cup P^{-n}(Z')\cup P^{-n-1}(Z')\cup P^{-n-1}(Z')|_{\neg ch}$, $|Z\cup P^{-n}(Z')\cup P^{-n-1}(Z')\cup P^{-n-2}(Z')|_{\neg ch}$ and so forth.\end{center} We also have infinitely many $\neg CH$-cardinal numbers between $|Z|_{\neg ch}$ and $|Z\cup P^{-n}(Z')|_{\neg ch}$, for example $|Z\cup P^{-n-1}(Z')|_{\neg ch}$, $|Z\cup P^{-n-2}(Z')|_{\neg ch}$ and so forth. An interesting question is to ask if there is a $\neg CH$-cardinal number between $|Z\cup P^{-n-1}(Z')|_{\neg ch}$ and $|Z\cup P^{-n}(Z')|_{\neg ch}$. This question, which relies on the definition of $\neg CH$-cardinality, might be undecidable in our setting. We will see in the following that it is possible to give an extension to the definition of cardinality which can give an answer to this question. Furthermore, to ask if there is always an extended cardinal number between two extended cardinal numbers could be seen as the real question behind the falsity of the continuum hypothesis. Seen as an analogy with numbers, it is similar to ask if between each pair of real numbers there is a real number.

We can also prove an extended version of Schroeder-Bernstein theorem for the $\neg CH$-cardinality.

\begin{proposition}If $|X|_{\neg ch}\leq |Y|_{\neg ch}$ and $|Y|_{\neg ch}\leq |X|_{\neg ch}$ then $|X|_{\neg ch}= |Y|_{\neg ch}$.\end{proposition}
\begin{proof}When $|X|_{\neg ch}\leq |Y|_{\neg ch}$ and $|Y|_{\neg ch}\leq |X|_{\neg ch}$, the only possibility is when $|X_1|=|Y_1|$. Since, if $|X_1|<|Y_1|$ this implies that $|Y|_{\neg ch}\nleq |X|_{\neg ch}$. By assumption, we have $m\leq n$, $n\leq m$, $\rho(X_i)\leq\rho(Y_i)$, $\tau(Y_i)\leq\tau(X_i)$, $\rho(X_i)\geq\rho(Y_i)$ and $\tau(Y_i)\geq\tau(X_i)$ for all $i$ and thus we have $|X_1|=|Y_1|$, $n=m$, $\rho(X_i)=\rho(Y_i)$ and $\tau(X_i)=\tau(Y_i)$. Therefore we can conclude that $|X|_{\neg ch}= |Y|_{\neg ch}$.\end{proof}

\subsection{Strong $\neg$Continuum Hypothesis}\label{negcontinuum}

Before giving the $\neg CHS$-cardinality definition, note that if $X=X_1\cup X_2 \cup ... \cup X_n$, $Y=Y_1\cup Y_2 \cup ... \cup Y_m$ and $n>m$, we can write $Y$ as $Y_1\cup Y_2 \cup ... \cup Y_n$, by adding a certain amount of union components which are empty sets.

\begin{definition}[$\neg CHS$-cardinality]\label{lexismaller} Let X, Y be sets of $\text{EZF}^*$. Let $X=X_1\cup X_2 \cup ... \cup X_n$ and let $Y=Y_1\cup Y_2 \cup ... \cup Y_n$ where $X_1,Y_1$ are Zermelo (possibly empty) and where $X_2,...,X_n, Y_2,...,Y_n$ are non-Zermelo or empty with $X_2\cup X_3 \cup ... \cup X_n$ and $Y_2\cup Y_3 \cup ... \cup Y_n$ both well-represented. Then, $|X|_{\neg chs}<|Y|_{\neg chs}$ if and only if

\medskip

\noindent  for some $k$ such that $k\leq n$, we have $\rho(X_i)=\rho(Y_i)$, $\tau(X_i)=\tau(Y_i)$, $\rho(X_k)<\rho(Y_k)$ and $\tau(X_k)\leq\tau(Y_k)$ for all $i<k$.

\medskip

\noindent or

\medskip

\noindent  for some $k$ such that $k\leq n$, we have $\rho(X_i)=\rho(Y_i)$, $\tau(X_i)=\tau(Y_i)$, $\rho(X_k)\leq\rho(Y_k)$ and $\tau(X_k)<\tau(Y_k)$ for all $i<k$ .
\end{definition}

\begin{definition}\label{equalnegcard}Using the same assumptions as in definition \ref{lexismaller}.

$|X|_{\neg chs}=|Y|_{\neg chs}$ if and only if \begin{center}for all $j\leq n$, we have $\rho(X_j)=\rho(Y_j)$ and $\tau(X_j)=\tau(Y_j)$.\end{center}
\end{definition}

The definition of $\neg CHS$-cardinality is basically the lexicographic order or, in analogy with numbers, it is similar to ordering of real numbers in decimal notation. Note that the transitivity of $<$ on the $\neg CHS$-cardinal numbers follows from the transitivity of $<$ on the integers and $<$ on the classical cardinal numbers.

\begin{proposition}If $|X|_{\neg chs}\leq|Y|_{\neg chs}$ and $|Y|_{\neg chs}\leq|Z|_{\neg chs}$, then $|X|_{\neg chs}\leq|Z|_{\neg chs}$.\end{proposition}
\begin{proof}Let X, Y, Z be sets of $\text{EZF}^*$. Let $X=X_1\cup X_2 \cup ... \cup X_n$, $Y=Y_1\cup Y_2 \cup ... \cup Y_n$ and let $Z=Z_1\cup Z_2 \cup ... \cup Z_n$ where $X_1,Y_1,Z_1$ are Zermelo sets (possibly empty) and where $X_2,...,X_n,Y_2,...,Y_n,Z_2,...,Z_n$ are non-Zermelo or empty with $X_2\cup X_3 \cup ... \cup X_n$, $Y_2\cup Y_3 \cup ... \cup Y_n$ and $Z_2\cup Z_3 \cup ... \cup Z_n$ are each well-represented. This can be done by adding empty union components to make sure that each $X,Y,Z$ have exactly $n$ components.

If $|X|_{\neg chs}=|Y|_{\neg chs}$ and $|Y|_{\neg chs}=|Z|_{\neg chs}$ then by definition \ref{equalnegcard}, we must have that $|X|_{\neg chs}=|Z|_{\neg chs}$.

If $|X|_{\neg chs}<|Y|_{\neg chs}$ and $|Y|_{\neg chs}=|Z|_{\neg chs}$ then there is a $k_1$ such that for all $i<k_1$, $\rho(X_i)=\rho(Y_i)$, $\tau(X_i)=\tau(Y_i)$ and such that $$[\rho(X_{k_1})<\rho(Y_{k_1})\wedge\tau(X_{k_1})\leq\tau(Y_{k_1})]\vee[\rho(X_{k_1})\leq\rho(Y_{k_1})\wedge\tau(X_{k_1})<\tau(Y_{k_1})].$$ Since $\tau(Y_j)=\tau(Z_j)$ and $\rho(Y_j)=\rho(Z_j)$ for all $j<n$, we can take $k=k_1$ and conclude, using the $\neg CHS$-cardinality definition, that $|X|_{\neg chs}<|Z|_{\neg chs}$. Similarly, we prove that if $|X|_{\neg chs}=|Y|_{\neg chs}$ and $|Y|_{\neg chs}<|Z|_{\neg chs}$, then $|X|_{\neg chs}<|Z|_{\neg chs}$.

If $|X|_{\neg chs}<|Y|_{\neg chs}$ and $|Y|_{\neg chs}<|Z|_{\neg chs}$, then we take $k_1$ for the $k$ which appears in the $\neg CHS$-cardinality definition for $|X|_{\neg chs}<|Y|_{\neg chs}$ and take $k_2$ for the $k$ which comes from $|Y|_{\neg chs}<|Z|_{\neg chs}$. Let $m=\min\{k_1, k_2\}$.

If $m=k_1$, this means that $k_2>m$ and that $\tau(Y_i)=\tau(Z_i)$ and $\rho(Y_i)=\rho(Z_i)$ for all $i<k_2$. Therefore, we have $\tau(X_m)<\tau(Y_m)=\tau(Z_m)$ or $\rho(X_m)<\rho(Y_m)=\rho(Z_m)$ which implies by the $\neg CH$-cardinality definition that $|X|_{\neg chs}<|Z|_{\neg chs}$. In a similar manner, if $m=k_2$, then we find that $\tau(X_m)=\tau(Y_m)<\tau(Z_m)$ or $\rho(X_m)=\rho(Y_m)<\rho(Z_m)$ which implies that $|X|_{\neg chs}<|Z|_{\neg chs}$.
\end{proof}

We now prove the stronger version of the negative of the continuum hypothesis. A good way to make the analogy with real numbers is to think of the superscript negative integers of the inverse powerset operator as the position of the decimal in the decimal expansion of a real number. Also, we can think of the cardinal number returned by the function $\tau$ as the value of that same decimal.

\begin{theorem}[$\neg$ Continuum Hypothesis]Let $X,Y$ be sets of $\text{EZF}^*$ such that $|X|_{\neg chs}<|Y|_{\neg chs}$, then there is a set $U$ of $\text{EZF}^*$ such that $|X|_{\neg chs}<|U|_{\neg chs}<|Y|_{\neg chs}$.\end{theorem}

\begin{proof}Let X, Y be sets of $\text{EZF}^*$. Let $X=X_1\cup X_2 \cup ... \cup X_n$ and let $Y=Y_1\cup Y_2 \cup ... \cup Y_n$ where $X_1,Y_1$ are Zermelo (possibly empty) and where $X_2,...,X_n, Y_2,...,Y_n$ are non-Zermelo or empty with $X_2\cup X_3 \cup ... \cup X_n$ and $Y_2\cup Y_3 \cup ... \cup Y_n$ both well-represented.

Since $|X|_{\neg chs}<|Y|_{\neg chs}$, there is a $k$ such that [$\rho(X_k)\leq\rho{Y_k}$ and $\tau(X_k)<\tau{Y_k}$] or [$\rho(X_k)<\rho{Y_k}$ and $\tau(X_k)\leq\tau{Y_k}$]. Let $r$ be the minimum integer (different from $-\infty$) of $\{\rho(X_1), ... , \rho(X_n),\rho(Y_1), ... , \rho(Y_n)\}$. Suppose that the smallest cardinal number (different from $|\emptyset|$) of $\tau(X_1), ... , \tau(X_n),\tau(Y_1), ... , \tau(Y_n)$ is $\tau(Z)$ such that $\tau(Z)=|K'|$ where $K'$ is a Zermelo set. If $Z$ is non-Zermelo take $K=K'$ and if $Z=P(P(...P(K'')...))$ where $K''$ is a non-powered Zermelo set $K$, take $K=K''$.

Take $U=X\cup P^{r-1}(K)$. We will show that $|X|_{\neg chs}<|X\cup P^{r-1}(K)|_{\neg chs}$ and $|X\cup P^{r-1}(K)|_{\neg chs}<|Y|_{\neg chs}$.

In $X=X_1\cup X_2 \cup ... \cup X_n$ there are finitely many empty $X_h$ components. Let $X=X_1\cup X_2 \cup ... \cup X_m$ where $m\leq n$ and where all $X_i$ are not empty. Since $\rho(P^{r-1}(K))<\rho(X_j)$ for all $j\leq m$ , we have that $X\cup P^{r-1}(K)=X_1\cup X_2 \cup ... \cup X_m\cup P^{r-1}(K)$ is well-presented. Thus, we have that $X_i=X_i$ for all $i\leq m$, $0<\rho(P^{r-1}(K))$ and $0<\tau(P^{r-1}(K))$ for $k=m+1$ which means, by definition \ref{lexismaller} of $\neg CHS$-cardinality, that $|X|_{\neg chs}<|X\cup P^{r-1}(K)|_{\neg chs}$.

Similarly, we have $X\cup P^{r-1}(K)=X_1\cup X_2 \cup ... \cup X_m\cup P^{r-1}(K)$ where all $X_i$ are non-empty. In $Y=Y_1\cup Y_2 \cup ... \cup Y_n$ there are finitely many empty components. Let $Y=Y_1\cup Y_2 \cup ... \cup Y_{m'}$ where $m'\leq n$ and where all $Y_i$ are not empty. There are two cases to consider, either $m<m'$ or $m\geq m'$.

Suppose that $m\geq m'$. For an analogy with numbers, this would be the case where $X=1.563$ and $Y=1.570$, so that $1.5630<1.5631<1.5700$. Since $|X|_{\neg chs}<|Y|_{\neg chs}$ there is some $k$ which satisfies a statement of definition \ref{lexismaller}. Thus, since $\rho(P^{r-1}(K))<\rho(X_j)$ for all $j$, adding the union component $P^{r-1}(K)$ to $X$ cannot make $|X\cup P^{r-1}(K)|_{\neg chs}\geq|Y|_{\neg chs}$, therefore we have that $|X\cup P^{r-1}(K)|_{\neg chs}<|Y|_{\neg chs}$.

 Suppose that $m<m'$. For an analogy with numbers, this would be the case where $X=1.5630$ and $Y=1.5631$, so that $1.56300<1.56301<1.56310$. Since $\rho(P^{r-1}(K))<\rho(Y_i)$ for all $i\leq n$ we can take $k=m+1$ and have that $\rho(P^{r-1}(K))<\rho(Y_k)$ and $\tau(P^{r-1}(K))\leq\tau(Y_k)$. Thus, by definition \ref{lexismaller} of $\neg CH$-cardinality we have that
$|X\cup P^{r-1}(K)|_{\neg chs}<|Y|_{\neg chs}$.
\end{proof}

\section{Further Investigations}

Recall that the sets of $\text{EZF}^*$ come from a finite number of applications of the operator $P^{-1}$. It might be worthwhile to investigate objects of the form $P^{-1}(P^{-1}(...P^{-1}(z)...))$ where $P^{-1}$ occurs an uncountable infinite number of times.

The full extension of ZF with the concept of the inverse powerset has not been fully investigated here since we restricted our context to $\text{EZF}^*$. It would be interesting to extend many of the axioms of ZF to include the non-Zermelo sets. In particular, we could further extend the powerset axiom to apply to sets such as $P^{-1}(\{1,2,3,4,5\}\cup P^{-1}(P^{-1}(X)))$, $\{P^{-1}(X),Y\}$ and $P^{-1}(X) \cup P^{-1}(Y)\cup Z \cup P^{-1}(P^{-1}(W))$.

A `fundamental theorem of set theory' resembling the fundamental theorems of arithmetic and algebra could be a very useful tool in set theory. An idea related to extensions of $\text{EZF}^*$ would be to investigate if it is possible to find such a fundamental theorem of set theory where each object of the extension can be represented uniquely in the form
$$X\cup P^{-1}(Y_1)\cup P^{-1}(Y_2)\cup ...\cup P^{-1}(Y_{n_1})\cup ...\cup P^{-k}(Z_1)\cup P^{-k}(Z_2)\cup ...\cup P^{-k}(Z_{n_k}),$$
\noindent where $X,Y_1,Y_2,...,Y_{n_1},...,Z_1,Z_2,...,Z_{n_k}$ can also be written in a certain form. Such a statement would imply in the realm of cardinality, something similar to the uniqueness of representation of numbers in decimal notation.

The continuum hypothesis is undecidable in $\text{ZF}$ and in $\text{ZF}$ with the axiom of choice, but extending to $\text{EZF}$ and $\text{EZF}^*$ permitted us to decide the continuum hypothesis. A key feature about $\text{EZF}$ is that it can be seen as `algebraically closed' under the powerset. To which extent can we diminish the quantity of undecidable statements by setting those statements in the context of the `algebraic closure' of a certain theory? Could the number of undecidable statements be reduced to a finite amount?

It might also be interesting to generalize the powerset operation to an operation which takes as input a set $X$ of $\text{ZF}$ (or $\text{EZF}$) and outputs only a certain collection of subsets of $X$. This contrasts with the powerset $P(X)$ which is an operation which returns a collection of every subset of $X$. It would be interesting to see the impact of this weakening of the concept of powerset and inverse powerset to the arithmetic of cardinal numbers. Furthermore, we could consider objects such as $P^n(X)$ where $n$ can be a rational number or a real number.

Finally, since most of mathematics relies on set theory, further investigations could be to consider introducing the non-Zermelo sets in the context of different mathematical theories such as topology and metric spaces.

\subsection*{Acknowledgments} The author would like to thank Johanna Okker, F. William Lawvere, Peter Aczel, Justin Moore, Greg Restall and Dana Scott for their valuable comments and suggestions.

\end{document}